\documentclass[10pt,reqno]{article}
\RequirePackage[OT1]{fontenc}
\RequirePackage{amsthm,amsmath}
\RequirePackage[numbers]{natbib}
\RequirePackage[colorlinks,linkcolor=blue,citecolor=blue,urlcolor=blue]{hyperref}
\usepackage{amssymb}
\usepackage{anysize}
\usepackage{extarrows}
\usepackage{multirow}
\usepackage{natbib}
\usepackage{stmaryrd}
\usepackage{color}
\usepackage{amsmath}
\usepackage{enumitem}
\usepackage{mathtools}
\usepackage{dsfont} 
\usepackage{comment}
\usepackage{soul}

\numberwithin{equation}{section}
\DeclareMathOperator{\tr}{Tr}

\DeclareMathOperator{\red}{red}
\DeclareMathOperator{\ev}{ev}

\DeclareMathOperator{\id}{id}

\let\liminf\relax
\DeclareMathOperator*{\liminf}{liminf}

\newcommand{\norm}[1]{\left\Vert #1\right\Vert}

\newtheorem{theorem}{Theorem} [section]
\newtheorem{prop}[theorem]{Proposition} 
\newtheorem{defi}[theorem]{Definition} 
 
\newtheorem{lemma}[theorem]{Lemma} 
\newtheorem{rem}[theorem]{Remark}

\newcommand\N{\mathbb{N}}
\newcommand\1{\mathbf{1}}
\newcommand\C{\mathbb{C}}
\newcommand\CC{\mathcal{C}}
\newcommand\cstar{\mathcal{C}^*}

\newcommand\R{\mathbb{R}}

\newcommand\A{\mathcal{A}}

\newcommand\F{\mathbb{F}}

\newcommand{\deq}{\mathrel{\mathop:}=}

\allowdisplaybreaks

\begin{document}

\begin{minipage}{0.85\textwidth}
	\vspace{2.5cm}
\end{minipage}
\begin{center}
	\Large\bf A Haagerup inequality through the use of orthogonal polynomials
	
\end{center}

\renewcommand{\thefootnote}{\fnsymbol{footnote}}	
\vspace{0.8cm}

\begin{center}
	F\'elix Parraud\\
	\footnotesize 
	{Queen's university}\\
	{\it felix.parraud@gmail.com}
\end{center}

\bigskip
\bigskip

\begin{abstract}

In this paper we prove that the Haagerup inequality for non-homogeneous polynomials in free semicircular variables of degree $n$ is optimal with a constant of order $n^{3/2}$. We also show an operator valued Haagerup inequality which improves on existing results. Our main tool to do so are free Chebyshev polynomials also known as $0$-Hermite polynomials.
	
\end{abstract}

\section{Introduction}

The notion of Haagerup inequality was first invented by Haagerup himself who discovered that the operator norm of the elements of a certain $\cstar$-algebra could be controlled by its $L^2$-norm up to a surprisingly small constant. More precisely, let $\cstar_{\red}(\mathbb{F}_d)$ be the reduced $\cstar$-algebra of $\F_d$ the free group with $d$ generators. He proved in \cite{haagineq} that for any $n\in\N$,
$$ \norm{\sum_{|g|=n} \alpha_g \lambda(g)}_{\cstar_{\red}(\mathbb{F}_d)} \leq (n+1) \norm{\sum_{|g|=n} \alpha_g \lambda(g)}_{L^2(\F_d)}, $$
where $\alpha_g\in\C$, $\lambda(g)$ is a unitary operator associated to $g$ and $|g|$ is the reduced length of the word $g$. Given that for $d\geq 2$, the number of words of such length grows exponentially with $n$, one would not expect the ratio of the $L^{\infty}$-norm over the $L^2$-norm to be bounded by a constant growing linearly in $n$. Using this inequality, Haagerup proved that $\cstar_{\red}(\mathbb{F}_d)$ has the metric approximation property. 

Since the original paper of Haagerup, these kind of inequalities have been found to be a lot more common than one would expect it at first. Groups whose reduced group $\cstar$-algebra satisfies those inequalities are said to have the rapid decay property. We refer to \cite{chat} for a history of the problem. Besides, this property turned out to have numerous application in different fields, such as operator algebras, noncommutative harmonic analysis, and geometric group theory. See for example \cite{e1,e2,e3,e31,haagineq,e4,e5,e6}, or more recently \cite{f3} for some applications to $q$-Gaussians, and \cite{f1,f2} for a broad range of application to the theory of $\CC^*$-algebras.

This paper focuses on the case of free semicircular variables. A Haagerup inequality for this model was first proved in \cite{bo} before being proved again in \cite{bs} as a corollary. This result was then extended to the case of $q$-Gaussians in \cite{boq} and \cite{krolak}. In particular, those results all have in common that the ratio of the operator norm and the $L^2$-norm of a homogeneous polynomial in free semicircular variables of degree $n$ is bounded by $n+1$ which is known to be optimal since this inequality is an equality for a Chebyshev polynomial of the second kind of degree $n$. On the contrary, for non-homogeneous polynomial, at best one can deduce that the ratio of the operator norm and the $L^2$-norm can be upper bounded by a quantity of order $n^{3/2}$. In the following theorem we prove that this is in fact optimal up to a constant independent of $n$.

\begin{theorem}
	\label{main1}
	
	There exists a family of orthogonal projections $(P_n)_{n\geq 0}$ such that for any given $z\in L^2(x)$, if $\sum_{n\geq 0} (n+1) \norm{P_nz}_2$ is finite, then $z\in\CC^*(x)$, and 
	\begin{equation}
		\label{skjdvnslknvd}
		\norm{z} \leq \sum_{n\geq 0} (n+1) \norm{P_nz}_2.
	\end{equation}
	In particular, $P_0+\dots+P_n$ is the orthogonal projection on the set of polynomials of degree at most $n$. Hence for any such polynomials $P$,
	\begin{equation}
		\label{svlknsldn}
		\norm{P(x)} \leq \sqrt{\frac{1}{3}\Big(n+1\Big)\left(n+\frac{3}{2}\right)\Big(n+2\Big)} \norm{P(x)}_2.
	\end{equation}
	Besides one can find polynomials $Q_n$ of degree $n$ such that 
	\begin{equation}
		\label{lskdnvclsnv}
		\liminf_{n\to\infty} \frac{\norm{Q_n(x)}}{\sqrt{\frac{1}{3}\Big(n+1\Big)\left(n+\frac{3}{2}\right)\Big(n+2\Big)} \norm{Q_n(x)}_2} \geq \sqrt{\frac{3}{8}} .
	\end{equation}
\end{theorem}

Note that in the case of homogeneous polynomials in circular variables, and more generally in $R$-diagonal elements, one can replace the quantity $n+1$ in Equation \eqref{skjdvncslkdm} by $\sqrt{n}$ up to a constant independent of $n$. This was first done in \cite{KS}. However since in this case one can only consider polynomials in circular variables $c_1,\dots,c_d$ but not in $c_1,\dots,c_d,c_1^*,\dots,c_d^*$, one cannot use this theorem to improve Theorem \ref{main1}. Note that this result was improved to families of operators satisfying weaker assumptions in \cite{brannan} and to the case of $q$-circular variables in \cite{KM}. 

Another direction that Theorem \ref{main1} can be improved into is by considering polynomials with operator coefficients instead of scalar coefficients. In the self-adjoint case, this was mainly investigated in \cite{delas} where the author proved among other things that given self-adjoint free copies $(x_k)_{1\leq k \leq d}$ of an operator $x$ with a symmetric compactly supported spectral measure, then for some constant $K(x)$ independent of every parameter,
\begin{equation}
	\label{skjdvnlasn}
	\norm{\sum_{i_1,\dots,i_n\in\N} a_{i_1,\dots,i_n}\otimes x_{i_1} \dots x_{i_d}} \leq K(x) (n+1) \max_{0\leq l\leq n} \norm{M_l},
\end{equation}
where $a_{i_1,\dots,i_n}$ are matrices such that $a_{i_1,\dots,i_n}=0$ as soon as $i_s=i_{s+1}$ for some $s$, and $M_l$ are matrices built out of those coefficients. Note in particular that one can take $x$ to be a semicircular variable. In this case however we prove the following theorem.

\begin{theorem}
	\label{main2}
	
	Given $H$ a Hilbert space, $a_{i_n,\dots,i_1}\in B(H)$ the space of bounded linear operator on $H$, one has that	
	$$ \norm{\sum_{i_n,\dots,i_1\in [1,d]} a_{i_n,\dots,i_1}\otimes P_{i_n,\dots,i_1}(x)} \leq \sum_{j=0}^n \norm{ \left(a_{(I,J)}\right)_{I\in [1,d]^{n-l}, J\in [1,d]^{l}}}, $$
	where $\left(a_{(I,J)}\right)_{I\in [1,d]^{n-l}, J\in [1,d]^{l}}$ is a matrix going from $H^{ d^l}$ to $H^{ d^{n-l}}$ with coefficients in $B(H)$, and $P_{i_n,\dots,i_1}$ is as in Definition \ref{skjdvncslkdm} (see also Proposition \ref{sjdvnslndv} for a more explicit equivalent definition).
	
\end{theorem}

Note that in the case of semicircular variables, this improve upon Equation \eqref{skjdvnlasn} in two directions. First, one does not need to assume that $i_s\neq i_{s+1}$. Indeed since the family $ \mathcal{F} = (P_{i_m,\dots,i_1}(x))_{i_1,\dots,i_m\in [1,d], 1\leq m\leq n} $ is an orthonormal basis of the set of polynomials in $x$ of degree smaller or equal to $n$, given a (non-necessarily homogeneous) polynomial with operator coefficients, one can decompose every monomial in the orthonormal basis $\mathcal{F}$ and then use Theorem \ref{main2}. In particular, note that if for all $s$, $i_s\neq i_{s+1}$ then $P_{i_n,\dots,i_1}(x) = x_{i_n} \dots x_{i_1}$.

Secondly, since the matrices $M_l$ are precisely the matrices $\left(a_{(I,J)}\right)_{I\in [1,d]^{n-l}, J\in [1,d]^{l}}$, upper bounding by the sum of their norm is naturally sharper than upper bounding by $n+1$ times the maximum of their norm. In fact, if the coefficients are scalar, then since
$$ \norm{ \left(a_{(I,J)}\right)_{I\in [1,d]^{n-l}, J\in [1,d]^{l}}} \leq \sqrt{\tr_N\left( \left|\left(a_{(I,J)}\right)_{I\in [1,d]^{n-l}, J\in [1,d]^{l}}\right|^2 \right)} = \norm{\sum_{i_n,\dots,i_1\in [1,d]} a_{i_n,\dots,i_1} P_{i_n,\dots,i_1}(x)}_2,$$
Theorem \ref{main2} yields an even sharper bound than Theorem \ref{main1}. For a concrete example, we refer to the discussion following Theorem 0.3 of \cite{delas}.

Finally, note that Theorem \ref{main2} can be viewed as an operator-valued non-commutative Khintchine inequality. Indeed, in the case of the free group $\F_d$, it was proved in \cite{HP} that
$$ \norm{\sum_{i=1}^{d} a_i\otimes \lambda(g_i) } \leq 2 \max\left(\norm{\sum_{i=1}^{d} a_i^*a_i},\norm{\sum_{i=1}^{d} a_ia_i^*}\right).$$
See also \cite{CMi} for the case of a free semicircular system. Hence Theorem \ref{main2} proves a similar upper bound but for homogeneous polynomials of any degree instead of only $1$.

\section{Preliminaries}

In order to be self-contained, we begin by recalling the following definitions  from free probability.

\begin{defi}~
	\label{3freeprob}
	\begin{itemize}
		\item A linear map $\tau : \A \to \C$ defined on a unital $\mathcal{C}^*$-algebra $(\A,*,\norm{.})$ is said to be a state if it satisfies $\tau(1_{\A})=1$ and $\tau(a^*a)\geq 0$ for all $a\in \A$. Besides, $\tau$ is said to be a trace, if $\tau(ab) = \tau(ba) $ for any $a,b\in\A$. A trace is said to be faithful if $\tau(a^*a)=0$ if and only if $a=0$.
		
		\item  Let $\A_1,\dots,\A_n$ be $*$-subalgebras of $\A$, having the same unit as $\A$. They are said to be free if for all $k$, for all $a_i\in\A_{j_i}$ such that $j_1\neq j_2$, $j_2\neq j_3$, \dots , $j_{k-1}\neq j_k$:
		\begin{equation}
			\label{kddkdxkfl}
			\tau\Big( (a_1-\tau(a_1))(a_2-\tau(a_2))\dots (a_k-\tau(a_k)) \Big) = 0.
		\end{equation}
		Families of operators are said to be free if the $*$-subalgebras they generate are free.
		
		\item A $d$-tuple of operators $ x=(x_1,\dots ,x_d)$ is called a free semicircular system if they are free, self-adjoint ($x_i=x_i^*$), and for all $k$ in $\N$ and $i$, one has
		\begin{equation*}
			\tau( x_i^k) =  \int_{\R} t^k d\sigma(t),
		\end{equation*}
		with $d\sigma(t) = \frac 1 {2\pi} \sqrt{4-t^2} \ \mathbf 1_{|t|\leq2} \ dt$ the semicircle distribution. The full Fock space gives us an explicit construction for a $\CC^*$-algebra endowed with a faithful trace which contains a system of free semicircular variables. For a full introduction to the topic, we refer to Chapter 7 of \cite{nica_speicher_2006}.
		
		\item We denote by $\CC^*(x)$ the $\mathcal{C}^*$-algebra generated by the free semicircular system $x$. Note that for $a,b\in\CC^*(x)$, the map $\langle x,y \rangle \deq \tau(xy^*)$ is a scalar product. We denote by $L^2(x)$ the completion of $\CC^*(x)$ for this scalar product and denote by $\norm{\cdot}_2$ the norm that it induces.
		
	\end{itemize}
	
\end{defi}

Let us finally fix a few notations concerning the spaces and traces that we use in this paper.

\begin{defi} \
	\begin{itemize}
		\item $\C\langle X_1,\dots,X_d\rangle$ is the space of non-commutative polynomials in $d$ variables.
		\item $\1$ will be the unit of the $*$-algebra $\C\langle X_1,\dots,X_d\rangle$.
	\end{itemize}
\end{defi}

Let us now define the so-called non-commutative derivative. This notion was introduced long ago, and is widely used in free probability, see, e.g., \cite{v98} for the inception of the concept. 
\begin{defi}
	\label{1application}
	
	If $1\leq i\leq d$, one defines the non-commutative derivative $\partial_i: \C\langle X_1,\dots,X_d\rangle \longrightarrow \C\langle X_1,\dots,X_d\rangle^{\otimes 2}$ on a monomial  $M$ by
	$$ \partial_i M = \sum_{M=AX_iB} A\otimes B ,$$
	and then extend it by linearity to all polynomials. 
		
\end{defi}

Those operators are related to the Schwinger-Dyson equation for a semicircular system thanks to the following Property \ref{1SDE}. One can find a proof in \cite{alice}, Lemma 5.4.7.

\begin{prop}
	\label{1SDE}
	Let $ x=(x_1,\dots ,x_p)$ be a free semicircular system whose $\mathcal{C}^*$-algebra they generate is endowed with a trace $\tau$. Then for any $P\in \C\langle X_1,\dots,X_d\rangle$,
	$$ \tau(P(x)\ x_i) = \tau\otimes\tau(\partial_i P(x))\ .$$	
\end{prop}

Finally let us conclude this section with the following proposition which states that free semicircular variables are algebraically free. For a proof, we refer to Proposition 2.4 of \cite{maispeichcomp}.

\begin{prop}
	\label{freealg}
	Given $x$ a free semicircular system, then for any polynomials $P\in \C\langle X_1,\dots,X_d\rangle$, if $P(x)=0$, then $P=0$.
\end{prop}

\section{Definition and properties of the free Chebyshev polynomials}

Chebyshev polynomials are a well-known tool of approximation theory that were first introduced by Chebyshev himself in 1853. There exist several equivalent definitions of those polynomials. However our approach will be different from the usual one. We begin by defining a higher order non-commutative derivative.

\begin{defi}
	\label{sdvjnskdvn}
	We define for $i_1,i_2,\dots,i_n\in [1,d]$, and $M\in\C\langle X_1,\dots,X_d\rangle$ a monomial,
	$$ \partial_{i_1,i_2,\dots,i_n} M = \sum_{M=A_1 X_{i_1} A_2\dots A_n X_{i_n} A_{n+1}} A_1 \otimes A_2\otimes\dots\otimes A_n \otimes A_{n+1}. $$
	Then we extend $\partial_{i_1,i_2,\dots,i_n}$ to $\C\langle X_1,\dots,X_d\rangle$ by linearity.
\end{defi}

Next we define our free Chebyshev polynomials through the following induction formula. By explicit computations one can check that in the classical case (i.e. when $i_1=\cdots=i_n=j$ for a given $j\in[1,d]$) those polynomials are simply Chebyshev polynomials of the second kind evaluated in $X_j$ and re-scaled to be orthonormal with respect to the semicircle law with support on $[-2,2]$. It is worth noting that the induction formula given by this definition does not seem to match with the usual one in the classical case, this difference will be removed in Proposition \ref{sjdvnslndv}.

\begin{defi}
	\label{skjdvncslkdm}
	We define by induction $P_{\emptyset}=\1$ and
	$$ P_{i_n,\dots,i_1} = X_{i_n} P_{i_{n-1},\dots,i_1}  - \ev\otimes\id\left( \partial_{i_n} P_{i_{n-1},\dots,i_1} \right),$$
	where for any given polynomial $Q$ we set $\ev(Q) \deq \tau(Q(x))$ with $x$ a $d$-tuple of free semicircular variables and $\tau$ the trace on the $\CC^*$-algebra they generate.
\end{defi}

The reason for this definition is the following proposition that states that our free Chebyshev polynomials are actually the adjoints of the higher order derivatives of Definition \ref{sdvjnskdvn} evaluated in $\1^{\otimes n+1}$.

\begin{prop}
	\label{kdsjnskdnvs}
	Given $x$ a $d$-tuple of free semicircular variables and $\tau$ the trace on the $\CC^*$-algebra they generate, then for any polynomials $Q$,
	$$ \tau\left( P_{i_n,\dots,i_1}(x) Q(x) \right) = \tau^{\otimes n+1} \left(\partial_{i_1,i_2,\dots,i_n}Q(x)\right). $$
	Thus $P_{i_n,\dots,i_1}(x) = (\partial_{i_1,i_2,\dots,i_n})^*(\1^{\otimes n+1})$.
\end{prop}

\begin{proof}
	
	Let us proceed by induction. The case where $n=1$ is simply the Schwinger-Dyson equations, i.e. Proposition \ref{1SDE}. If $n>1$, thanks once again to Proposition \ref{1SDE}, one has
	\begin{align*}
		\tau\left(x_{i_n} P_{i_{n-1},\dots,i_1}(x) Q(x) \right) &= \tau\otimes\tau\left(\partial_{i_n}(P_{i_{n-1},\dots,i_1}Q)(x) \right) \\
		&= \tau\otimes\tau\left( \partial_{i_n} P_{i_{n-1},\dots,i_1}(x) \ \1\otimes Q(x)  \right) + \tau\otimes\tau\left(P_{i_{n-1},\dots,i_1}(x)\otimes\1 \ \partial_{i_n}Q(x)  \right) \\
		&= \tau\left( \ev\otimes\id\left(\partial_{i_n} P_{i_{n-1},\dots,i_1}(x)\right) Q(x)  \right) + \tau\otimes\tau\left(P_{i_{n-1},\dots,i_1}(x)\otimes\1 \ \partial_{i_n}Q(x)  \right) .
	\end{align*}
	Consequently, one has that
	\begin{align*}
		\tau\left( P_{i_n,\dots,i_1}(x) Q(x) \right) &= \tau\otimes\tau\left(P_{i_{n-1},\dots,i_1}(x)\otimes\1 \ \partial_{i_n}Q(x)  \right) \\
		&= \tau^{\otimes n+1} \left(\partial_{i_1,i_2,\dots,i_n}Q(x)\right),
	\end{align*}
	where we used our induction hypothesis in the last line.
	
\end{proof}

The next proposition states that the adjoint of a free Chebyshev polynomial is also a Chebyshev polynomial.

\begin{prop}
	\label{skdjvns}
		For any $i_1,i_2,\dots,i_n\in [1,d]$, one has
		$$ P_{i_n,\dots,i_1}^* = P_{i_1,\dots,i_n}. $$
\end{prop}

\begin{proof}
	For any monomial $M$, thanks to Proposition \ref{kdsjnskdnvs} one has 
	\begin{align*}
		\tau\left(P_{i_n,\dots,i_1}^*(x) M(x)\right) &= \overline{\tau\left(P_{i_n,\dots,i_1}(x) M^*(x)\right)} \\
		&= \overline{ \tau^{\otimes n+1} \left(\partial_{i_1,i_2,\dots,i_n}M^*(x)\right) } \\
		&=  \sum_{M^*=A_1 X_{i_1} A_2\dots A_n X_{i_n} A_{n+1}} \overline{\tau(A_1(x))} \dots \overline{\tau(A_{n+1}(x))} \\
		&=  \sum_{M= A_{n+1}^* X_{i_n} A_n^*\dots A_2^* X_{i_1} A_1^* } \tau(A_{n+1}^*(x)) \dots \tau(A_1^*(x)) \\
		&= \tau^{\otimes n+1} \left(\partial_{i_n,i_{n-1},\dots,i_1}M(x)\right) \\
		&= \tau\left(P_{i_1,\dots,i_n}(x) M(x)\right).
	\end{align*}
	Consequently, by linearity, for any polynomial $Q$, 
	$$ \tau\left(\left(P_{i_n,\dots,i_1}^*(x) - P_{i_1,\dots,i_n}(x)\right) Q(x)\right)=0.$$
	Thus by orthogonality $P_{i_n,\dots,i_1}^*(x) = P_{i_1,\dots,i_n}(x)$, hence the conclusion thanks to Proposition \ref{freealg}.
\end{proof}

It is well-known that Chebyshev polynomials form an orthonormal basis for the $L^2$-space associated with the semicircle distribution on $[-2,2]$. It turns out that our definition of free Chebyshev polynomials generalizes this property to $L^2(x)$.

\begin{prop}
	\label{ksdncslmdlv}
	The family $(P_{i_n,\dots,i_1}(x))_{i_1,\dots,i_n\in [1,d], n\in\N}$ is an orthonormal basis of $L^2(x)$ where $x$ is a $d$-tuple of free semicircular variables.
\end{prop}

\begin{proof}
	By a straightforward induction, one has that  
	$$ P_{i_n,\dots,i_1} = X_{i_n}\dots X_{i_1} + Q, $$
	where $Q$ is a polynomial of degree strictly smaller than $n$. Thus $(P_{i_n,\dots,i_1})_{i_1,\dots,i_n\in [1,d], n\in\N}$ is a basis of $\C\langle X_1,\dots,X_d\rangle$. Since the set of polynomials evaluated in $x$ is dense in $L^2(x)$, there only remains to prove the orthonormality of our family. Let us assume that $n\geq m$, thanks to Proposition \ref{kdsjnskdnvs} and \ref{skdjvns}, one has
	\begin{align*}
		\tau\left( P_{i_n,\dots,i_1}(x) P_{j_m,\dots,j_1}^*(x)  \right) &= \tau^{\otimes n+1}\left( \partial_{i_1,i_2,\dots,i_n} P_{j_1,\dots,j_m}	(x)  \right) \\
		&= \tau^{\otimes n+1}\left( \partial_{i_1,i_2,\dots,i_n} (x_{j_1}\dots x_{j_m})  \right) \\
		&= \1_{n=m} \1_{i_1=j_1,\dots,i_n=j_n}.
	\end{align*}
	Hence the conclusion since if $n<m$,
	$$ \tau\left( P_{i_n,\dots,i_1}(x) P_{j_m,\dots,j_1}^*(x) \right) = \overline{ \tau\left( P_{j_m,\dots,j_1}(x) P_{i_n,\dots,i_1}^*(x) \right)} =0. $$
\end{proof}	

It is known (see for example Section 2.3 of \cite{jamieroland}) that classical Chebyshev polynomials behave surprisingly well with respect to the non-commutative differential. Indeed if $P_n$ is the $n$-th of those polynomials, then
$$ \partial_1 P_n = \sum_{l=0}^{n-1} P_l\otimes P_{n-1-l}.$$
Once again, free Chebyshev polynomials satisfy a similar formula.

\begin{prop}
	\label{sdincvsncd}
	For any $i_1,i_2,\dots,i_n\in [1,d]$, one has
	$$ \partial_k P_{i_n,\dots,i_1} = \sum_{j=1}^n \1_{i_j=k}\ P_{i_n,\dots,i_{j+1}}\otimes P_{i_{j-1},\dots,i_1}. $$
\end{prop}

\begin{proof}
	
	Let us begin by induction. It is clear if $n=1$ since
	$$ \partial_kP_{i_1} = \partial_kX_{i_1} = \1_{i_1=k}\ P_{\emptyset}\otimes P_{\emptyset}. $$
	If $n>1$, then with $s\leq n$, thanks to our induction hypothesis,
	$$ P_{i_s,\dots,i_1} = X_{i_s} P_{i_{s-1},\dots,i_1}  - \sum_{l=1}^{s-1} \1_{i_l=i_s} \tau\left( P_{i_{s-1},\dots,i_{l+1}}(x)\right) P_{i_{l-1},\dots,i_1} .$$
	But thanks to Proposition \ref{ksdncslmdlv}, if $l<s-1$, then by orthogonality,
	$$\tau\left( P_{i_{s-1},\dots,i_{l+1}}(x)\right) = \tau\left( P_{i_{s-1},\dots,i_{l+1}}(x) P_{\emptyset}^*\right) = 0.$$
	Consequently,
	\begin{equation}
		\label{sdkcns}
		P_{i_s,\dots,i_1} = X_{i_s} P_{i_{s-1},\dots,i_1}  - \1_{i_{s-1}=i_s} P_{i_{s-2},\dots,i_1}.
	\end{equation}
	Hence
	\begin{align*}
		\partial_k P_{i_n,\dots,i_1} &= \1_{i_n=k}\ P_{\emptyset} \otimes P_{i_{n-1},\dots,i_1} + \sum_{j=1}^{n-1} \1_{i_j=k}\ X_{i_n}P_{i_{n-1},\dots,i_{j+1}}\otimes P_{i_{j-1},\dots,i_1} \\
		&\quad  - \1_{i_{n-1}=i_n} \sum_{ j=1 }^{n-2} \1_{i_j=k} P_{i_{n-2},\dots,i_{j+1}}\otimes P_{i_{j-1},\dots,i_1} \\
		&= \sum_{j=1}^{n-2} \1_{i_j=k} \left( X_{i_n}P_{i_{n-1},\dots,i_{j+1}} - \1_{i_{n-1}=i_n} P_{i_{n-2},\dots,i_{j+1}} \right)\otimes P_{i_{j-1},\dots,i_1} \\
		&\quad + \1_{i_n=k}\ P_{\emptyset} \otimes P_{i_{n-1},\dots,i_1} + \1_{i_{n-1} = k} P_{i_n} \otimes P_{i_{n-2},\dots,i_1} \\
		&= \sum_{j=1}^n \1_{i_j=k}\ P_{i_n,\dots,i_{j+1}}\otimes P_{i_{j-1},\dots,i_1},
	\end{align*}
	where we used Equation \eqref{sdkcns} in the last line.
	
\end{proof}

The proof above actually also yields the following result.

\begin{prop}
	\label{sjdvnslndv}
	The family of polynomials $(P_{i_n,\dots,i_1})_{n\in\N, i_j\in [1,d]} $ of Definition \ref{sdvjnskdvn} satisfies the following induction formula, for $n=0$, $P_{\emptyset}=\1$, for $n=1$, $P_{i_1}=X_{i_1}$, and for $n>1$,
	\begin{equation*}
		P_{i_n,\dots,i_1} = X_{i_n} P_{i_{s-1},\dots,i_1}  - \1_{i_{n-1}=i_n} P_{i_{s-2},\dots,i_1}.
	\end{equation*}
\end{prop}

This induction formula matches with the classical case. Indeed, if $P_n$ is the $n$-th Chebyshev polynomials of the second kind and re-scaled to be orthonormal with respect to the semicircle law with support on $[-2,2]$. Then it is known (see for example Section 2.3 of \cite{jamieroland}) that
$$P_n(X) = XP_n(X) - P_{n-1}(X),$$
which is the same formula as the one given by the Proposition above. In fact the induction defined in the proposition above is the one of $q$-Hermite polynomials for $q=0$, see section 2 of \cite{qher}. That being said, neither Definition \ref{skjdvncslkdm} nor Proposition \ref{sjdvnslndv} will be of much use when manipulating free Chebyshev polynomials. Instead we will mostly use Proposition \ref{ksdncslmdlv}, \ref{sdincvsncd} as well as \ref{djkfnvslnc}.

It is known (see for example Section 2.3 of \cite{jamieroland}) that multiplying classical Chebyshev polynomials is a simple operation. More precisely, one has that
$$ P_nP_m= P_{n+m} + P_{n+m-2} + \dots + P_{|n-m|}. $$
Once again this property generalizes to the free case.

\begin{prop}
	\label{djkfnvslnc}
	Given $i_1,\dots,i_n,j_1,\dots, j_m\in [1,d]$, then
	$$ P_{i_n,\dots,i_1} P_{j_m,\dots, j_1} = \sum_{s=0}^{m\wedge n} \1_{(i_1,\dots,i_s) = (j_m,\dots,j_{m-s+1})} P_{i_n,\dots,i_{s+1},j_{m-s},\dots,j_1}. $$
\end{prop}

\begin{proof}
	
	Thanks to Proposition \ref{freealg}, it is enough to show that 
	$$ P_{i_n,\dots,i_1}(x) P_{j_m,\dots, j_1}(x) = \sum_{s=0}^{m\wedge n} \1_{(i_1,\dots,i_s) = (j_m,\dots,j_{m-s+1})} P_{i_n,\dots,i_{s+1},j_{m-s},\dots,j_1}(x).$$
	Besides, thanks to Proposition \ref{kdsjnskdnvs} and \ref{skdjvns},
	\begin{align*}
		\tau\left( P_{i_n,\dots,i_1}(x) P_{j_m,\dots, j_1}(x) P_{k_l,\dots, k_1}^*(x)  \right) &= \tau\left( P_{i_n,\dots,i_1}(x) P_{j_m,\dots, j_1}(x) P_{k_1,\dots, k_l}(x)  \right) \\
		&= \tau^{\otimes l+1} \left(\partial_{k_l,\dots,k_1}(P_{i_n,\dots,i_1} P_{j_m,\dots, j_1})(x)\right).
	\end{align*}
	Thus necessarily for this quantity not to be zero, one must assume that $l\leq n+m$. Besides, let us remember that for any $u_n,\dots,u_1\in [1,d]$, $P_{u_n,\dots,u_1}$ is orthogonal to $P_{\emptyset}$, and hence
	$$\tau\left(P_{u_n,\dots,u_1}(x)\right)=0.$$
	Consequently, 
	\begin{align*}
		&\tau\left( P_{i_n,\dots,i_1}(x) P_{j_m,\dots, j_1}(x) P_{k_l,\dots, k_1}^*(x)  \right) \\
		&= \sum_{\substack{r_1\in [0,n],r_2\in[0,m], \\ r_1+r_2=l}} \1_{(i_n,\dots,i_{n-r_1+1})=(k_l,\dots,k_{l-r_1+1})} \1_{(j_{r_2},\dots,j_1)=(k_{r_2},\dots,k_1)} \tau\left( P_{i_{n-r_1},\dots,i_1} P_{j_m,\dots, j_{r_2+1}}(x)\right) \\
		&= \sum_{\substack{r_1\in [0,n],r_2\in[0,m], \\ r_1+r_2=l}} \1_{ (k_l,\dots,k_1) = (i_n,\dots,i_{n-r_1+1},j_{r_2},\dots,j_1)} \tau\left( P_{i_{n-r_1},\dots,i_1} P_{j_{r_2+1},\dots, j_m}^*(x)\right) \\
		&= \sum_{\substack{r_1\in [0,n],r_2\in[0,m], \\ r_1+r_2=l}} \1_{ (k_l,\dots,k_1) = (i_n,\dots,i_{n-r_1+1},j_{r_2},\dots,j_1)} \1_{n-r_1 = m-r_2} \1_{(i_{n-r_1},\dots,i_1) = (j_{r_2+1},\dots, j_m)} \\
		&= \sum_{s= 0}^{n\wedge m} \sum_{\substack{r_1\in [0,n],r_2\in[0,m], \\ r_1+r_2=l}} \1_{ (k_l,\dots,k_1) = (i_n,\dots,i_{s+1},j_{m-s},\dots,j_1)}  \1_{s=n-r_1 = m-r_2} \1_{(i_{s},\dots,i_1) = (j_{m-s+1},\dots, j_m)}\\
		&= \1_{(i_{s_l},\dots,i_1) = (j_{m-s_l+1},\dots, j_m)} \1_{ (k_l,\dots,k_1) = (i_n,\dots,i_{s_l+1},j_{m-s_l},\dots,j_1)}  \1_{\frac{m+n-l}{2} \in [0,n\wedge m]\cap \N},
	\end{align*}

	\noindent where $s_l \deq \frac{m+n-l}{2}$. In the last line we used that the system of equation
	$$\left\{ \begin{array}{rl}
		s\ =& n-r_1 \\
		s\ =& m-r_2 \\
		l\ =& r_1+r_2
	\end{array}\right.$$
	has exactly one solution which is 
	$$ s = \frac{m+n-l}{2},\quad r_1= \frac{n+l-m}{2},\quad r_2= \frac{m+l-n}{2}. $$
	Besides, since we assumed that $l\leq n+m$, this solution is such that $r_1\in [0,n],r_2\in[0,m], s\in [0,n\wedge m]$ if and only if $\frac{m+n-l}{2} \in [0,n\wedge m]\cap \N$. Consequently, for $\tau\left( P_{i_n,\dots,i_1}(x) P_{j_m,\dots, j_1}(x) P_{k_l,\dots, k_1}^*(x)  \right)$ to be equal to one, there must exists an integer $s\in [0,n\wedge m]$ such that $(i_1,\dots,i_s) \allowbreak = (j_m,\dots,j_{m-s+1})$ and $(k_l,\dots,k_1) = (i_n,\dots,i_{s+1},j_{m-s},\dots,j_1)$. 
	
	Reciprocally if this condition is satisfied then since $l$ is the cardinal of the family $(k_l,\dots,k_1)$, one must have $l=n-s+m-s$, hence $s_l=s$ is an integer, $(i_1,\dots,i_{s_l}) \allowbreak = (j_m,\dots,j_{m-s_l+1})$ and $(k_l,\dots,k_1) = (i_n,\dots,i_{s_l+1},j_{m-s_l},\dots,j_1)$, hence $\tau\left( P_{i_n,\dots,i_1}(x) P_{j_m,\dots, j_1}(x) P_{k_l,\dots, k_1}^*(x)  \right)=1$.
	
	Consequently,	
	$$ P_{i_n,\dots,i_1}(x) P_{j_m,\dots, j_1}(x) = \sum_{s=0}^{m\wedge n} \1_{(i_s,\dots,i_1) = (j_m,\dots,j_{m-s+1})} P_{i_n,\dots,i_{s+1},j_{m-s},\dots,j_1}(x).$$
	Hence the conclusion.
\end{proof}

\begin{rem}
	Note that in particular, if $j_m\neq i_1$, then $P_{i_n,\dots,i_1} P_{j_m,\dots, j_1} = P_{i_n,\dots,i_1,j_m,\dots, j_1}$. Consequently, given $i_1,\dots,i_k$ such that for all $j\in [1,k-1]$, $i_j\neq i_{j+1}$, one has that
	$$ P_{i_1,\dots,i_1,i_2,\dots,i_2,i_3,\dots,i_{k-1},i_k,\dots,i_k}(X_1,\dots,X_d) = P_{d_1}(X_{i_1})\dots P_{d_k}(X_{i_k}),$$
	where $d_j$ is the number of times where $i_j$ is repeated and $P_n$ is the the $n$-th Chebyshev polynomials of the second kind re-scaled to be orthonormal with respect to the semicircle law with support on $[-2,2]$. Once again, even though this remark tells us a lot about the structure of our free Chebyshev polynomials, it will be of little use when manipulating them. Instead we will mostly use Proposition \ref{ksdncslmdlv}, \ref{sdincvsncd} as well as \ref{djkfnvslnc}.
\end{rem}

\section{The Haagerup inequality}

The proof of this section should be compared to the proof of the original Haagerup inequality in \cite{haagineq}. Interestingly, the proof of Lemma \ref{skjdvnsknc} and \ref{sldncslms} below are very similar to the one of Lemma 1.3 and 1.4 of \cite{haagineq} respectively. From this one can deduce that studying the $\CC^*$-algebra generated by a free semicircular system in the orthonormal basis provided by the free Chebyshev polynomials allows us to use tools and strategies developed for the reduced $\CC^*$-algebra generated by a free group.

Let us start by giving a few definitions.

\begin{defi}
	Given $f\in\C^*(x)$, $i_1,\dots,i_n\in [1,d]$, we will denote
	$$ f(i_n,\dots,i_1) = \tau(f\ (P_{i_n,\dots,i_1}(x))^*).$$
	Besides, one will say that $f$ is homogeneous of degree $n$ if $f$ is a linear combination of $(P_{i_n,\dots,i_1}(x))_{i_1,\dots,i_n\in [1,d]}$. This is equivalent to assuming that $f(i_m,\dots,i_1) = 0$ as soon as $m\neq n$.
\end{defi}

\begin{defi}
	We denote $P_n$ the orthogonal projection on the vector space generated by $(P_{i_n,\dots,i_1}(x))_{i_1,\dots,i_n\in [1,d]}$.
\end{defi}

The notion of an element being homogeneous of degree $n$ is the semicircular equivalent of assuming that an element of a reduced $\CC^*$-algebra generated by a free group is a linear combination of words in its generator of reduced length $n$. This notion will be important in the rest of the paper.

\begin{lemma}
	\label{skjdvnsknc}
	Given $l,m,n$ non-negative integers, $f$ and $g\in\C^*(x)$ homogeneous of degree $l$ and $m$ respectively, then if $l+m-n$ is even and $n\in [|l-m|,l+m]$,
	$$ \norm{P_n fg}_2 \leq \norm{f}_2\norm{g}_2,$$
	and otherwise 	$\norm{P_n fg}_2 =0$.	
\end{lemma}

\begin{proof}
	
	We have thanks to Proposition \ref{djkfnvslnc} that 
	$$ (fg)(k_n,\dots,k_1) = \sum_{\substack{ i_l,\dots,i_1,j_m,\dots,j_1\in[1,d] \\ \text{such that } \exists s\in[0,l\wedge m],\ (k_n,\dots,k_1)= (i_l,\dots,i_{s+1},j_{m-s},\dots,j_1) \\ \text{and } (i_1,\dots, i_s)=(j_{m},\dots,j_{m-s+1}) }} f(i_l,\dots,i_1) g(j_m,\dots,j_1). $$
	Note that if $i_l,\dots,i_1,j_m,\dots,j_1\in[1,d]$ are such that there exists $s\in[0,l\wedge m]$ such that $ (k_n,\dots,k_1)= (i_l,\dots,i_{s+1},j_{m-s},\dots,j_1)$, then necessarily $n=l+m-2s$. In particular, for $(fg)(k_n,\dots,k_1)$ to not be zero, $l+m-n$ must be even, and $n$ must belong to $[|l-m|,l+m]$. Hence $\norm{P_n fg}_2 =0$ if $n$ does not satisfy one of those conditions.
	
	Assume now that $n=l+m$, then $(fg)(k_n,\dots,k_1) = f(k_n,\dots,k_{m+1}) g(k_m,\dots,k_1)$, and
	$$ \norm{P_n fg}_2 = \sum_{ i_l,\dots,i_1,j_m,\dots,j_1\in[1,d]} |f(i_l,\dots,i_1)|^2 |g(j_m,\dots,j_1)|^2 = \norm{f}_2\norm{g}_2.$$
	Let us now assume that $n\in [|l-m|,l+m-2]$ and that $l+m-n$ is even. Let us denote $p=\frac{l+m-n}{2}$, we define 
	$$ f'(a_q,\dots,a_1) = \left(\sum_{ b_p,\dots,b_1\in [1,d] } |f(a_q,\dots,a_1,b_p,\dots,b_1)|^2\right)^{1/2} \text{ if } q=l-p \text{ and } f'(a_q,\dots,a_1)=0 \text{ else},$$
	$$ g'(a_q,\dots,a_1) = \left(\sum_{ b_p,\dots,b_1\in [1,d] } |g(b_p,\dots,b_1,a_q,\dots,a_1)|^2\right)^{1/2} \text{ if } q=m-p \text{ and } g'(a_q,\dots,a_1)=0 \text{ else.}$$
	Note that
	$$ \norm{f'}_2 = \sum_{ a_q,\dots,a_1 \in [1,d] } \sum_{ b_p,\dots,b_1\in [1,d] } |f(a_q,\dots,a_1,b_p,\dots,b_1)|^2 = \norm{f}_2, $$
	and similarly $\norm{g'}_2 = \norm{g}_2$. Thus
	\begin{align*}
		|(fg)(k_n,\dots,k_1)| &= \left|\sum_{\substack{ i_l,\dots,i_1,j_m,\dots,j_1\in[1,d] \\ \text{such that } (k_n,\dots,k_1)= (i_l,\dots,i_{p+1},j_{m-p},\dots,j_1) \\ \text{and } (i_1,\dots, i_p)=(j_{m},\dots,j_{m-p+1}) }} f(i_l,\dots,i_1) g(j_m,\dots,j_1) \right| \\
		&= \left|\sum_{b_p\dots,b_1\in[1,d]} f(k_n,\dots,k_{m-p+1},b_p\dots,b_1) g(b_1,\dots,b_p, k_{m-p},\dots,k_1) \right| \\
		&\leq |f'(k_n,\dots,k_{m-p+1})| \times |g'(k_{m-p},\dots,k_1)| \\
		&= |(f'g')(k_n,\dots,\dots,k_1)|,
	\end{align*}
	where the last equality follows because $f'$ and $g'$ are homogeneous of degree $l-p$ and $m-p$ respectively and that $n=l+m-2p$. For this same reason, one has that
	$$\norm{ P_n f'g' }_2 = \norm{ f'}_2\norm{g' }_2 = \norm{ f}_2\norm{g}_2.$$
	Hence,
	$$\norm{ P_n fg }_2 \leq \norm{ P_n f'g' }_2 = \norm{ f}_2\norm{g}_2.$$
\end{proof}

\begin{lemma}
	\label{sldncslms}
	Let $f\in \CC^*(x)$ be a polynomial in $x$, then 
	$$ \norm{f} \leq \sum_{n\geq 0} (n+1) \norm{P_nf}_2. $$
\end{lemma}

\begin{proof}
	
	Given that $(P_{i_n,\dots,i_1}(x))_{i_1,\dots,i_n\in [1,d], n\in [0,g]}$ is an orthonormal basis of the set of polynomials of degree smaller than $g$, one has
	$$ \norm{f} \leq \sum_{l\geq 0} \norm{ P_lf }. $$
	Thus one can assume that $f$ is homogeneous of degree $l$. Let $g$ be an arbitrary element of $L^2(x)$, let us set $g_m=P_mg$, then
	$$ \norm{g}^2 = \sum_{m\geq 0} \norm{g_m}_2^2.$$
	Consequently, Lemma \ref{skjdvnsknc} shows that if $l+m-n$ is even and $n\in [|l-m|,l+m]$, 
	$$ \norm{P_n fg_m}_2 \leq \norm{f}_2\norm{g_m}_2,$$
	and otherwise $\norm{P_n fg_m}_2=0$. Note that $n\in [|l-m|,l+m]$ if and only if $m\in [|l-n|,l+n]$, and that $l+m-n$ is even if and only if $l+n-m = l + m-n + 2(n-m)$ is even. Therefore,
	\begin{align*}
		\norm{P_n fg}_2 &\leq \sum_{m\geq 0}\norm{P_n fg_m}_2 \\
		&\leq \norm{f}_2 \sum_{\substack{m\in[|l-n|,l+n],\\l+n-m \text{ is even}}} \norm{g_m}_2 \\
		&= \norm{f}_2 \sum_{k=0}^{\min(l,n)} \norm{g_{l+n-2k}}_2 \\
		&\leq \norm{f}_2 \left(\sum_{k=0}^{\min(l,n)} \norm{g_{l+n-2k}}_2^2\right)^{1/2} \left(\min(l,n)+1\right)^{1/2} \\
		&\leq \norm{f}_2 \left(\sum_{k=0}^{\min(l,n)} \norm{g_{l+n-2k}}_2^2\right)^{1/2} \left(l+1\right)^{1/2}.
	\end{align*}
	Consequently, one has that
	\begin{align*}
		\norm{fg}_2^2 &= \sum_{n\geq 0} \norm{P_nfg}_2^2 \\
		&\leq \left(l+1\right) \norm{f}_2^2\ \sum_{n\geq 0} \sum_{k=0}^{\min(l,n)} \norm{g_{l+n-2k}}_2^2 \\
		&= \left(l+1\right) \norm{f}_2^2\ \sum_{k= 0}^l \sum_{n\geq k} \norm{g_{l+n-2k}}_2^2 \\
		&\leq \left(l+1\right) \norm{f}_2^2\ \sum_{k= 0}^l \sum_{n\geq 0} \norm{g_{n}}_2^2 \\
		&= \left(l+1\right)^2 \norm{f}_2^2 \norm{g}_2^2 \\
	\end{align*}
	Consequently,
	$$\norm{f} = \sup_{g\in L^2(x)} \frac{\norm{fg}_2}{\norm{g}_2} \leq \left(l+1\right) \norm{f}_2.$$
	Hence the conclusion.
	
\end{proof}

This allows us to prove our first theorem.

\begin{proof}[Proof of Theorem \ref{main1}]
	
	Let us first prove Equation \eqref{skjdvnslknvd}. We set $z_n=(P_0+\dots+P_n)z$. Thus thanks to Lemma \ref{sldncslms}, $(z_n)_{n\geq 0}$ is a Cauchy sequence in $\CC^*(x)$, hence converges towards a given $\widetilde{z}\in \CC^*(x)$. However for any $n$,
	$$ \norm{z-\widetilde{z}}_2 \leq \norm{z-z_n}_2 + \norm{z_n-\widetilde{z}}$$
	By definition, $\lim_{n\to\infty} \norm{z_n-\widetilde{z}} = 0$. And thanks to Parseval's identity, 
	$$ \norm{z-z_n}_2^2 = \sum_{k>n} \norm{P_kz}_2^2.$$
	In particular  $\lim_{n\to\infty} \norm{z-z_n}_2 = 0$. Consequently, $z=\widetilde{z}$ is an element of $\CC^*(x)$. Besides by using Lemma \ref{sldncslms} one more time,
	$$ \norm{z} = \norm{\widetilde{z}} = \lim_{n\to\infty} \norm{z_n} \leq \lim_{n\to\infty} \sum_{k=0}^{n} (k+1) \norm{P_kz}_2. $$
	Hence the conclusion. The proof of Equation \eqref{svlknsldn} is simply a consequence of the fact that $ (P_0+\dots+P_n)P(x) = P(x)$, hence $P_kP(x)= 0$ for $k>n$, hence by Cauchy-Schwarz and Parseval's identity,
	$$ \norm{P(x)} \leq \sum_{k= 0}^n (k+1) \norm{P_nP(x)}_2 \leq \sqrt{\sum_{k=0}^{n} (k+1)^2 \sum_{k=0}^{n} \norm{P_nP(x)}_2^2} = \sqrt{\frac{(n+1)(n+2)(2n+3)}{6}} \norm{P(x)}_2. $$
	Finally, if we set $R_n=P_{1,\dots,1}$ where $1$ is repeated $n$ times, then
	$$ R_{n+1}(2) = 2R_n(2) - R_{n-1}(2),$$
	and since one has that $R_0(2)=1$, and $R_1(2)=2$, by an immediate induction we get that $R_n(2)=n+1$. Hence if we set $Q_{2n} = R_n^2$, one has thanks to Proposition \ref{djkfnvslnc} that $ Q_{2n} = \sum_{k=0}^n R_{2k}$, thus $\norm{Q_{2n}(x)}_2 = \sqrt{n+1}$ and
	$$\frac{\norm{P_{2n}(x)}}{\sqrt{\frac{(2n+1)(2n+2)(4n+3)}{6}} \norm{P_{2n}(x)}_2} \geq \frac{(n+1)^{3/2}}{\sqrt{\frac{(2n+1)(2n+2)(4n+3)}{6}} } = \sqrt{\frac{3}{8}} \frac{n+1}{\sqrt{(n+\frac{1}{2})(n+\frac{3}{4})} }. $$
	We also set $P_{2n+1}= Q_nQ_{n+1}$, and with similar computations we get Equation \eqref{lskdnvclsnv}. 
	
\end{proof}

\section{The operator valued version}

If the proof of the previous section was to be compared to the original proof of the Haagerup inequality in \cite{haagineq}, then the one of this section should be compared with \cite{bush}. Once again, the orthonormal basis provided by the free Chebyshev polynomials allows us to use tools and strategies developed for the unitary case.

\begin{defi}
	Given $i_n,\dots,i_1\in [1,d]$, we define $U_{i_n,\dots,i_1}: L^2(x)\mapsto L^2(x)$ by setting
	$$ U_{i_n,\dots,i_1}\left(P_{k_m,\dots,k_1}(x)\right) \deq P_{i_n,\dots,i_1,k_m,\dots,k_1}(x). $$
\end{defi}

We immediately have the following lemma.

\begin{lemma}
	Given $i_n,\dots,i_1\in [1,d]$, then $U_{i_n,\dots,i_1}^*: L^2(x)\mapsto L^2(x)$ is defined by
	$$ U_{i_n,\dots,i_1}^*\left(P_{k_m,\dots,k_1}(x)\right) = \left\{  \begin{array}{cc} 0 & \text{ if } (k_m,\dots,k_{m-n+1})\neq (i_n,\dots,i_1), \\
		P_{k_{m-n},\dots,k_1} & \text{ else.}
	\end{array} \right. $$
	Besides $U_{i_n,\dots,i_1}$ is an isometry such that $ U_{i_n,\dots,i_1}^*U_{i_n,\dots,i_1}$ is the identity over $L^2(x)$ and  $ U_{i_n,\dots,i_1}U_{i_n,\dots,i_1}^*$ is the projection on the space generated by the following orthonormal basis,
	$$ \left\{  P_{k_m,\dots,k_1}(x)\ \middle|\ (k_m,\dots,k_{m-n+1})= (i_n,\dots,i_1) \right\}. $$
\end{lemma}

\begin{prop}
	\label{dnvls}
	Given $i_n,\dots,i_1\in [1,d]$, if one views $P_{i_n,\dots,i_1}(x)$ as an element of $B(L^2(x))$, one has that 
	$$ P_{i_n,\dots,i_1}(x) = \sum_{l=0}^n U_{i_n,\dots,i_{l+1}} U_{i_l,\dots,i_1}^*. $$
	
\end{prop}

\begin{proof}
	It is enough to prove of this equality when evaluated in an element of the orthonormal basis. But then this is simply Proposition \ref{djkfnvslnc}.
\end{proof}

\begin{prop}
	\label{sljkdv}
	Given $a_{i_n,\dots,i_1}\in B(H)$, one has that	
	$$ \norm{\sum_{i_n,\dots,i_1\in [1,d]} a_{i_n,\dots,i_1}\otimes U_{i_n,\dots,i_{l+1}} U_{i_l,\dots,i_1}^*} = \norm{ \sum_{I\in [1,d]^{n-l}, J\in [1,d]^{l} } a_{(I,J)}\otimes E_{I,J}}, $$
	where $E_{I,J}$ is the matrix going from $H^{ d^l}$ to $H^{d^{n-l}}$ whose only non-zero coefficient is the one indexed by $(I,J)$ and is equal to $1$.
\end{prop}

\begin{proof}
	One has that 
	\begin{align*}
		\norm{\sum_{i_n,\dots,i_1\in [1,d]} a_{i_n,\dots,i_1}\otimes U_{i_n,\dots,i_{l+1}} U_{i_l,\dots,i_1}^*}^2 &=\norm{\sum_{I\in [1,d]^{n-l}, J\in [1,d]^{l}} a_{(I,J)}\otimes U_I U_J^*}^2 \\
		&= \norm{\sum_{I,S\in [1,d]^{n-l}, J,T\in [1,d]^{l}} a_{(I,J)}a^*_{(S,T)}\otimes U_I U_J^* U_T U_S^*}.
	\end{align*}
	However, one has that
	$$ U_J^* U_T = \id_{L^2(x)} \1_{J=T}. $$
	Consequently, with $e_J$ the $J$-th element of $\C^{d^l}$, if we fix $R\in[1,d]^l$, using that for any $X$ in a $\C^*$-algebra, $\norm{X}=\norm{X\otimes e_Re_R^*}$, one has that 
	\begin{align*}
		\norm{\sum_{i_n,\dots,i_1\in [1,d]} a_{i_n,\dots,i_1}\otimes U_{i_n,\dots,i_{l+1}} U_{i_l,\dots,i_1}^*}^2
		&= \norm{\sum_{I,S\in [1,d]^{n-l}, J,T\in [1,d]^{l}} e_{J}^*e_T\times a_{(I,J)}a^*_{(S,T)}\otimes U_I U_J^* \otimes e_Re_R^*} \\
		&= \norm{\sum_{I,S\in [1,d]^{n-l}, J,T\in [1,d]^{l}} a_{(I,J)}a^*_{(S,T)}\otimes U_I U_J^* \otimes e_R e_{J}^*(e_R e_T^*)^*} \\
		&=\norm{\sum_{I\in [1,d]^{n-l}, J\in [1,d]^{l}} a_{(I,J)}\otimes U_I \otimes e_Re_J^*}^2 \\
		&= \norm{\sum_{I,S\in [1,d]^{n-l}, J,T\in [1,d]^{l}} a_{(I,J)}^*a_{(S,T)}\otimes U_I^* U_S \otimes e_{J}e_T^*} \\
		&= \norm{\sum_{I,S\in [1,d]^{n-l}, J,T\in [1,d]^{l}} a_{(I,J)}^*a_{(S,T)}\otimes (e_I^*e_S \id_{L^2(x)}) \otimes e_{J}e_T^*} \\
		&= \norm{\sum_{I,S\in [1,d]^{n-l}, J,T\in [1,d]^{l}} a_{(I,J)}^*a_{(S,T)}\otimes \id_{L^2(x)} \otimes e_{J}e_I^*e_Se_T^*} \\
		&= \norm{\sum_{I,S\in [1,d]^{n-l}, J,T\in [1,d]^{l}} a_{(I,J)}^*a_{(S,T)} \otimes E_{I,J}^*E_{S,T}} \\
		&=\norm{\sum_{I\in [1,d]^{n-l}, J\in [1,d]^{l}} a_{(I,J)}\otimes E_{I,J}}.
	\end{align*}
	Hence the conclusion. 
\end{proof}

This allows us to prove our second theorem.

\begin{proof}[Proof of Theorem \ref{main2}]
	
	One has thanks to Proposition \ref{dnvls} that
	\begin{align*}
		\norm{\sum_{i_n,\dots,i_1\in [1,d]} a_{i_n,\dots,i_1}\otimes P_{i_n,\dots,i_1}(x)} &=  \norm{ \sum_{l=0}^n  \sum_{i_n,\dots,i_1\in [1,d]} a_{i_n,\dots,i_1}\otimes U_{i_n,\dots,i_{l+1}} U_{i_l,\dots,i_1}^*} \\
		&\leq \sum_{l=0}^n \norm{ \sum_{i_n,\dots,i_1\in [1,d]} a_{i_n,\dots,i_1}\otimes U_{i_n,\dots,i_{l+1}} U_{i_l,\dots,i_1}^*}.
	\end{align*}
	Hence the conclusion thanks to Proposition \ref{sljkdv}.

\end{proof}

\bibliographystyle{abbrv}

\end{document}